\documentclass[a4paper, 12pt]{amsart}

\usepackage[english]{babel}
\usepackage[T1]{fontenc}
\usepackage[applemac]{inputenc}
\usepackage{verbatim}
\usepackage{color}
\usepackage{amsmath}
\usepackage{amssymb}
\usepackage{amsthm}
\usepackage{psfrag}
\usepackage{graphicx}
\usepackage{braket}
\usepackage{geometry}
\usepackage{hyperref}
\usepackage[usenames,dvipsnames]{pstricks}
\usepackage{epsfig}
\usepackage{pst-grad} 
\usepackage{pst-plot} 

\usepackage{pstricks, pst-node, graphics}
\newpsobject{showgrid}{psgrid}{subgriddiv=1,griddots=10,gridlabels=6pt}

\usepackage{enumerate}

\renewcommand{\phi}{\varphi}
\newcommand{\Z}{\mathbb{Z}}
\newcommand{\Size}[1]{\left\lvert #1 \right\rvert}

\newcommand{\Span}[1]{\left\langle\, #1 \,\right\rangle}
\DeclareMathOperator{\End}{End}

\newcommand{\K}[0]{\mathcal{K}}
\newcommand{\Hc}[0]{\mathcal{H}}

\renewcommand{\theta}[0]{\vartheta}
\renewcommand{\phi}[0]{\varphi}

\newcommand{\F}{\mathbb{F}}
\DeclareMathOperator{\GL}{GL}
\DeclareMathOperator{\AGL}{AGL}
\DeclareMathOperator{\Alt}{Alt}
\DeclareMathOperator{\Sym}{Sym}
\DeclareMathOperator{\PSL}{PSL}
\DeclareMathOperator{\Aut}{Aut}
\DeclareMathOperator{\Out}{Out}
\newtheorem{theo}{Theorem}[section]

\newtheorem{remark}[theo]{Remark}

\newtheorem{lemma}[theo]{Lemma}
\newtheorem{definition}[theo]{Definition}

\newtheorem{cor}[theo]{Corollary}

\numberwithin{equation}{section}

\begin{document}

\date{28 July 2015, 10:16 CEST --- Version 7.12
}

\title[GOST]%
      {The group generated by\\ 
      the round functions of\\
      a GOST-like cipher}
      
\author{R. Aragona}

\address[R.~Aragona]%
 {Dipartimento di Matematica\\
  Universit\`a degli Studi di Trento\\
  via Sommarive 14\\
  I-38123 Trento\\
  Italy} 

\email{riccardo.aragona@unitn.it} 

\urladdr{http://science.unitn.it/$\sim$aragona/}

\author{A. Caranti}

\address[A.~Caranti]%
 {Dipartimento di Matematica\\
  Universit\`a degli Studi di Trento\\
  via Sommarive 14\\
  I-38123 Trento\\
  Italy} 

\email{andrea.caranti@unitn.it} 

\urladdr{http://science.unitn.it/$\sim$caranti/}

\author{M. Sala}

\address[M.~Sala]%
 {Dipartimento di Matematica\\
  Universit\`a degli Studi di Trento\\
  via Sommarive 14\\
  I-38123 Trento\\
  Italy} 

\email{massimilano.sala@unitn.it} 

\urladdr{http://science.unitn.it/$\sim$sala/}



\begin{abstract}        
 We  define a  cipher that  is  an extension  of GOST,  and study  the
 permutation group  generated by its  round functions.  We  show that,
 under minimal assumptions on the components of the cipher, this group
 is the alternating group on the plaintext space. This we do by first
 showing that the group is primitive, and then applying the
 O'Nan-Scott classification of primitive groups.
\end{abstract}

\thanks{The second author is grateful to the Department of Mathematics
  of the University of Trento for financial support. The first two
  authors are member of GNSAGA---Italy.}


\maketitle

\thispagestyle{empty}

\bibliographystyle{amsalpha}

\section{Introduction}

When DES  was about to  be broken by brute  force, and Triple  DES was
introduced as  a replacement,  Kaliski, Rivest and  Sherman considered
in~\cite{DES-group} the  question, whether  DES (that  is, the  set of
transformations it defines) is a group. Had this been the case, Triple
DES would have been no different from DES.  They gave evidence for the
fact that  DES was  indeed not a  group, and also  showed that  if the
group generated by  a cipher is too small, then  certain attacks based
on  the birthday  paradox are  possible. Note,  however, that  Murphy,
Paterson  and Wild  \cite{weak} have  constructed a  weak cipher  that
generates  the   whole  symmetric  group  ---   therefore  the  latter
requirement  alone is  not enough  to  guarantee the  strength of  the
cipher.

Coppersmith  and Grossman  defined a  set  of functions  which can  be
adapted for constructing  a block cipher, and  studied the permutation
group  generated  by  them~\cite{CoGr}.  Even  and  Goldreich  defined
certain  DES-like functions,  and  proved that  the permutation  group
generated     by     these     functions    is     the     alternating
group~\cite{EvGo}. Wernsdorf later showed  that the group generated by
the round functions of DES  is the alternating group~\cite{DES-W}, and
Sparr    and   Wernsdorf    showed   that    the   same    holds   for
KASUMI~\cite{KASUMI-W} and AES~\cite{AES-W}. Since the group generated
by a cipher (with independent round  keys) is a normal subgroup of the
group generated by the round functions, and the alternating group is a
simple group, it follows that the former group is also alternating.

In~\cite{us-1,  us-2, us-3}  another approach  to these  questions was
taken, in that  one first shows that the group  generated by the round
functions  of an  AES-like cipher  is a  primitive permutation  group,
provided the  S-boxes satisfy some cryptographic  assumptions, such as
being  weakly  APN functions.   This  shows  that  the cipher  has  no
imprimitivity   trapdoor~\cite{pat}.    And   then   the   O'Nan-Scott
classification of  finite primitive  groups~\cite{LPS, Li} is  used to
show that the  group must be alternating or symmetric.   In this paper
we  apply this  point  of view  to  an extension  of  the cipher  GOST
28147-89 \cite{CGC-cry-misc-dolmatov2010gost}, or  GOST for short, and
show  that its  round functions  generate the  alternating group.   It
might  be  noted that  we  require  only  minimal assumptions  on  the
components  of  this  cipher,  basically only  that  the  S-boxes  are
bijective,  and that  the  rotation has  the  ``right'' extent.   
This  appears to  indicate  that  the  Feistel structure plays an
important role in guaranteeing that the group is large.

Oliynykov  considered   in~\cite{oli2011nonbijective}  ciphertext-only
attacks  on Feistel  networks,  and  proved that  the  use of  secret,
non-bijective S-boxes  allows for  the introduction of  trapdoors.  In
particular, the author applied his results to GOST.

In Section~\ref{sec:1} we  describe GOST.
In  Section~\ref{sec:larger} we  introduce our  extension of  GOST. In
Section~\ref{sec:pri} we  show that the  group generated by  the round
functions    of   this    GOST-like    cipher    is   primitive.    In
Section~\ref{sec:ON-S}  we  analyse  the   cases  in  the  O'Nan-Scott
classification,  to conclude  that the  group generated  by the  round
functions of our GOST-like cipher is the alternating group.

\section{The group generated by the round functions of GOST}
\label{sec:1}

Consider the set $V^{0} = \F_{2}^{n}$, for some $n > 1$.  (Here
$\F_{2}$ is the field with two elements, and see
Remark~\ref{rem:valuesGOST} for the actual values in GOST of this, and
the other parameters we are going to introduce in the following.)  We
consider two group structures on $V^{0}$.  The first operation is the
bitwise sum (XOR), which will be denoted by $+$. The bitwise sum makes
$V^{0}$ into a vector space over $\F_{2}$.

The  second  operation,  denoted  by $\boxplus$,  is  the  sum  modulo
$2^{n}$. That is, we represent $a, b \in V^{0}$ as
\begin{equation*}
  a = (a_{0}, a_{1}, \dots, a_{n-1}),
  \quad
  b = (b_{0}, b_{1}, \dots, b_{n-1}),
\end{equation*}
with $a_{i}, b_{i} \in \Set{0, 1}$ integers, and let
\begin{equation*}
  a \boxplus b
  =
  (c_{0}, c_{1}, \dots, c_{n-1}),
\end{equation*}
where
\begin{multline*}
  (a_{0} + a_{1} 2 + a_{2} 2^{2} + \dots + a_{n-1} 2^{n-1})
  +
  (b_{0} + b_{1} 2 + b_{2} 2^{2} + \dots + b_{n-1} 2^{n-1})
  \equiv\\\equiv
  c_{0} + c_{1} 2 + c_{2} 2^{2} + \dots + c_{n-1} 2^{n-1}
  \pmod{2^{n}},
\end{multline*}
with $c_{i} \in  \Set{0, 1}$ integers. (Here $+$  denotes the ordinary
sum  of integers.)   Therefore $V^{0}$  under $\boxplus$  is the  same
thing as  the group  $\Z_{2^{n}}$ of integers  modulo $2^{n}$,  and we
will denote  it by $(\Z_{2^{n}},  \boxplus)$. We use $\boxminus  a$ to
indicate the opposite of $a \in V^{0}$ with respect to $\boxplus$.

We record a few elementary facts that we will be using repeatedly
without further mention.
\begin{lemma}\label{lemma:subgroups}
  \
  
  \begin{itemize}
  \item  The  subgroups  of   $(\Z_{2^{n}},  \boxplus)$  are  linearly
    ordered; they are the $\Span{2^{q}}$, for $0 \le q \le n$.
  \item The endomorphisms of $(\Z_{2^{n}},  \boxplus)$ are of the form
    $x \mapsto z \, x$, where $z$ is an integer, $0 \le z < 2^{n}$. Such a
    map is an automorphism if and only if $z$ is odd.
  \item Every subgroup of   $(\Z_{2^{n}},  \boxplus)$ is \emph{fully
    invariant} (that is, it is sent into itself by any endomorphism of
    $(\Z_{2^{n}},  \boxplus)$) 
    and thus \emph{characteristic} (that is, it is sent onto itself by any
    automorphism of   $(\Z_{2^{n}},  \boxplus)$).  
  \item 
    The element
    \begin{equation*}
      2^{n-1} = (0, 0, \dots, 0, 1)
    \end{equation*}
    is the only \emph{involution}  (that is, element of  order $2$)  of
    $(\Z_{2^{n}}, \boxplus)$.
    Therefore $2^{n-1}$ is fixed by any automorphism of
    $(\Z_{2^{n}},  \boxplus)$, and it is sent to zero by any
    endomorphism which is not an automorphism.
  \end{itemize}
\end{lemma}

In  GOST 28147-89  \cite{CGC-cry-misc-dolmatov2010gost} the  plaintext
space is $V = V^{1} \times V^{2}$, where $V^{1}, V^{2}$ are two copies
of $V^{0}$, and the  key   space   $\K$ is another copy of $V^{0}$.
Clearly $V$ inherits both group structures componentwise from $V^{1}, V^{2}$.

\begin{definition}\label{rem:values}
  When considering a subset of $V^{i}$, for $i = 0, 1, 2$, we will call it 
  \begin{itemize}
  \item   a
    \emph{subspace} if it is a subgroup (and thus a vector subspace) of
    $(\F_{2}^{n}, +)$, and 
  \item a $\boxplus$-\emph{subgroup}, or simply a \emph{subgroup}, if it
    is a subgroup of $(\Z_{2^{n}}, \boxplus)$.
  \end{itemize}
  This terminology can be extended to the subsets of $V$.
\end{definition}

\begin{definition}\label{brick}
 We will consider $V^{i}$, for $i = 0, 1, 2$, as the Cartesian product
 \begin{equation}\label{eq:directsum}
   V^{i}
   =
   V^{i}_{1} \times \cdots \times V^{i}_{\delta} 
   = 
   V^{i}_{1} \mid\mid \cdots \mid\mid V^{i}_{\delta}
 \end{equation}
 of $\delta  > 1$  subspaces $V^{i}_j$,  all of the  same dimension  $m >
 1$. (Here $\mid\mid$ denotes \emph{concatenation} of strings.)

 An element $\gamma$ of the  symmetric group $\Sym(V^{i})$ on $V^{i}$ is
 called    a    \emph{bricklayer    transformation}    with    respect
 to~\eqref{eq:directsum}   if  it   preserves   the  direct   product
 decomposition,    that    is,     if    there    are    \emph{S-boxes}
 $\gamma_j\in\mathrm{Sym}(V^{i}_{j})$ such that, writing $v \in V^{i}$ as
 \begin{equation*}
   v
   =
   (v_{1},  \cdots, v_{\delta}),
 \end{equation*}
 with $v_{j} \in V^{i}_{j}$, we have
 \begin{equation*}
    v \gamma 
    =
    (v_{1} \gamma_{1}, \cdots , v_{\delta} \gamma_{\delta}).
 \end{equation*}
 We will refer to each $V^{i}_{j}$ as a \emph{brick}. 
\end{definition}

Let $S = \gamma R \in  \Sym(V^{i})$, where $\gamma \in \Sym(V^{i})$ is
a bricklayer transformation and $R$ is  the right rotation by $r$ bits
(we refer  to $r$ as the  \emph{extent} of the rotation),  with $m\leq
r\leq (\delta-1)m$, that is
\begin{equation*}
  (a_0,\ldots,a_{n-1})R
  =
  (a_{n-r},\ldots,a_{n-1},a_{0},\ldots,a_{n-r-1}).
\end{equation*}

\begin{remark}\label{rem:valuesGOST}
In the case of  GOST, the actual values of the parameters are:
$n=32$, $m=4$, $\delta=8$ and $r=11$. 
\end{remark} 

For $(k_{1}, k_{2}) \in V = V^{1} \times V^{2}$, consider the 
$\boxplus$-translation on $V$ by $(k_{1}, k_{2})$
\begin{equation*}
\begin{array}{rccc}
  \rho_{(k_{1}, k_{2})} 
  :
  &V_{1} \times V_{2}
  &\longrightarrow &V_{1} \times V_{2}\\
  & (x_{1} , x_{2})
  &\longmapsto
  & (x_{1} \boxplus k_{1}, x_{2} \boxplus k_{2}).
\end{array}
\end{equation*}

We now introduce a formal $2n \times 2n$ matrix, which implements the
Feistel structure,
\begin{equation}\label{eq:Sigma}
  \Sigma = \begin{bmatrix}0 & 1\\ 1 & S\end{bmatrix},
\end{equation}
where $0$ and $1$ are $n \times n$ matrices.
This acts (on the right) on $(x_{1}, x_{2}) \in V = V^{1} \times V^{2}$ by
\begin{equation}\label{eq:Sigma-acts}
  (x_{1}, x_{2}) \Sigma
  =
  (x_{2}, x_{1} + x_{2} S).
\end{equation}
Note that $\Sigma$ has the formal inverse matrix
\begin{equation*}
  \Sigma^{-1}
  =
  \begin{bmatrix}
    S & 1\\ 
    1 & 0
  \end{bmatrix}.
\end{equation*}

A round function of GOST with respect  to the round key $k \in \K$ can
now be described as
\begin{equation}\label{eq:round}
  \tau_k 
  =
  \rho_{(0,k)} \, \Sigma \, \rho_{(\boxminus k,0)}.
\end{equation}
(As we  let permutations  act on  the right,  this is  a left-to-right
composition.)  In fact
\begin{align*}
  (x_{1}, x_{2}) \tau_{k}
  &=
  (x_{1}, x_{2}) \, \rho_{(0,k)} \,\Sigma \,\rho_{(\boxminus k,0)}
  \\&=
  (x_{1}, x_{2} \boxplus k)  \,\Sigma \,\rho_{(\boxminus k,0)}
  \\&=
  (x_{2} \boxplus k, x_{1} + (x_{2} \boxplus k) S) \, \rho_{(\boxminus k,0)}
  \\&=
  (x_{2}, x_{1} + (x_{2} \boxplus k) S).
\end{align*}

Thus the group generated by the round functions of GOST is
\begin{equation*}
  \mathcal{G} = \Span{ \tau_{k} : k\in\mathcal{K} }.
\end{equation*}

\section{A larger group}
\label{sec:larger}

In  our notation,  in an  actual GOST  round~\eqref{eq:round} the  key
addition   ($\boxplus$-translation)  preceding   $\Sigma$,  and   that
following $\Sigma$, are  related: the first one acts  only on $V^{2}$,
the  second  one  only  on  $V^{1}$,   and  the  extents  of  the  two
translations are  one the  $\boxplus$-opposite of  the other.  In this
paper  we  will be  studying  a  GOST-like  system  in which  a  round
generalizes the one of GOST:  we allow to $\boxplus$-sum two arbitrary
(unrelated)  pairs  of keys  before  and  after applying  the  Feistel
transformation $\Sigma$.  So in  our cipher the  plaintext $V$  is the
same as that of GOST, while the key space is $\Hc = \K \times \K = V$,
and a round takes the form

\begin{equation}\label{eq:round-like}
  \rho_{k} \, \Sigma \, \rho_{h},
\end{equation}
with $k, h \in \Hc$. Such a  round operates on $(x_{1}, x_{2}) \in V =
V^{1} \times V^{2}$ by
\begin{align*}
  (x_{1}, x_{2}) \, \rho_{k} \, \Sigma \, \rho_{ h}
  &=
  (x_{1} \boxplus k_{1}, x_{2} \boxplus k_{2})  \, \Sigma \,
  \rho_{h}
  \\&=
  (x_{2} \boxplus k_{2}, x_{1} \boxplus k_{1} + (x_{2} \boxplus
  k_{2}) S) \,
  \rho_{h}
  \\&=
  (x_{2} \boxplus k_{2} \boxplus h_{1}, 
  (x_{1} \boxplus k_{1} + (x_{2} \boxplus  k_{2}) S) \boxplus h_{2}),
\end{align*}
where $k_{i}, h_{i} \in V^{i}$.

The corresponding group will thus be
\begin{equation*}
  \Gamma
  =
  \Span{ 
    \rho_{k} \, \Sigma \, \rho_{h}
    :
    k, h \in \Hc
  }
\end{equation*}
Clearly our group $\Gamma$ contains the group $\mathcal{G}$ generated by the
round functions of GOST.

We collect a couple of elementary observations.
\begin{enumerate}
\item    $\Sigma\in\Gamma$. This follows from setting $k = h = 0$ 
  in~\eqref{eq:round-like}.
\item   For all $k \in \Hc$, we have that $\rho_{k} \in \Gamma$. It
  suffices to set $h = 0$ in~\eqref{eq:round-like} and then note that
  $\rho_{k} = (\rho_{k}\Sigma ) \Sigma^{-1}$ is in $\Gamma$, as both
  factors are. 
\end{enumerate}

Therefore
\begin{equation}\label{eq:gamma}
  \Gamma
  =
  \Span{
    \mathcal{T}, \Sigma
    },
\end{equation}
where
\begin{equation*}
  \mathcal{T} =
  \Set{
    \rho_{k}
    :
    k \in \Hc
    }
\end{equation*}
is the group of $\boxplus$-translations on $V$. In particular,
$\Gamma$ acts transitively on $V$,

We now state our main result. 

\begin{theo}\label{theo:sym}
  Let $n = \delta m$, with $\delta\geq 4$ and $m \geq 2$. Consider the
  $\F_{2}$-vector 
  spaces $V^{i} = \F_{2}^{n}$, for $i = 1, 2$, and $V = V^{1} \times
  V^{2}$, under the operation $+$. For $i = 1, 2$, write
  \begin{equation}\label{eq:bricks}
    V^{i} = V_{1}^{i} \times \dots \times V^{i}_{\delta},
  \end{equation}
  where each $V_{j}^{i}$ is a subspace of dimension $m$ over $\F_{2}$.
  
  Let $\boxplus$ be the operation on $V^{i}, V$ defined in the previous
  Section, so that each 
  $(V^{i}, \boxplus)$ is cyclic, of order $2^{n}$. Let $\mathcal{T}$ be
  the group of 
  $\boxplus$-translations $\rho_{k} : x \mapsto x \boxplus k$ on
  $V$, for $k \in V$.
  
  Consider
  \begin{enumerate}
  \item A bricklayer transformation $\gamma$ with respect
    to~\eqref{eq:bricks}. 
  \item 
    The right rotation $R$ by $r$ bits on $V^{i}$.
  \item $S = \gamma R$.
  \item The formal matrix
    \begin{equation*}
      \Sigma = 
      \begin{bmatrix}
        0 & 1\\ 
        1 & S
      \end{bmatrix},
    \end{equation*}
    which operates on $V = V^{1} \times V^{2}$ by
    \begin{equation*}
      (x_{1}, x_{2}) \Sigma
      =
      (x_{2}, x_{1} + x_{2} S).
    \end{equation*}
  \end{enumerate}

  Consider the GOST-like cipher with plaintext  and key space $V$,
  in which a round has the form
  \begin{equation*}
    \rho_{k} \, \Sigma \, \rho_{h},
  \end{equation*}
  for the round keys $k, h \in V$.

  Then the group generated by the round functions is
  \begin{equation*}
    \Gamma
    =
    \Span{
    \mathcal{T}, \Sigma
    },
  \end{equation*}
  where $\mathcal{T} = \Set{\rho_{k} : k \in V}$ is the set of
  $\boxplus$-translations on $V$. 
  
  Assume that
  \begin{enumerate}
  \item the rotation extent $r$ satisfies $m \le r \le (\delta - 1)
    m$, and
  \item the bricklayer transformation $\gamma$ is bijective
    (equivalently, each S-box is bijective, or $S$ is bijective).
  \end{enumerate}
  Then 
  \begin{equation*}
    \Gamma = \Alt(V).
  \end{equation*}
\end{theo}

Here $\Alt(V)$ is the alternating group, consisting of the even
permutations on the set $V$. We record the following
\begin{lemma}\label{lemma:even}
  All permutations of $\Gamma$ are even, that is, $\Gamma \le \Alt(V)$.
\end{lemma}

\begin{proof}
  The group $\mathcal T$ of $\boxplus$-translations is generated by
  $\rho_{(0,1)}$ and $\rho_{(1,0)}$. Both maps are even permutations,
  as each of them is the product of $2^{n}$ cycles of length $2^{n}$.

  We now show that $\Sigma$ is also an even permutation. $\Sigma$ can
  be considered as the composition of two permutations of order $2$ of
  $V$. The 
  first permutation
  \begin{equation*}
    (x_{1}, x_{2}) \mapsto (x_{2}, x_{1}),
  \end{equation*}
  which exchanges the coordinates, has the $2^{n}$ fixed points $(x,
  x)$, for $x \in V^{0}$,
  and thus it is the product of an even number
  \begin{equation*}
    \frac{2^{2n} - 2^{n}}{2} = 2^{2n-1} - 2^{n-1}
  \end{equation*}
  of $2$-cycles, as $n > 1$. The second permutation
  \begin{equation*}
    (x_{2}, x_{1}) \mapsto (x_{2}, x_{1} + x_{2} S)
  \end{equation*}
  has  also order  $2$, and  has also  $2^{n}$ fixed  points, which
  correspond to the
  value $x_{2} = 0 S^{-1}$, and thus it is also even.
\end{proof}

\begin{remark}
  The arguments of Section~\ref{sec:pri} could be extended to cover
  any rotation different from the identity. For the arguments of
  Subsection~\ref{sec:wreath} to work with any rotation different from
  the identity, however, we would need to add extra hypotheses on the
  behaviour of the last S-box. Therefore we have preferred to stick to
  this setting, which requires only two natural assumptions on the
  cipher.
\end{remark}

Let us consider a cipher consisting of  a fixed number of rounds as in
Theorem~\ref{theo:sym}  with   independent  round  keys.    The  group
$\Gamma'$ generated by (the transformations  of) this cipher will be a
normal subgroup of  $\Gamma$. (See Lemma~\ref{lemma:fromroundstowhole}
below.)   Since the alternating group acting on at least $5$ letters
is simple, it follows $\Gamma'$ 
is also the alternating group on $V$.

\begin{lemma}\label{lemma:fromroundstowhole}
  Let $\Gamma$ be a group generated by elements $g_{i}$, for some
  index set.

  Let $N$ be a positive integer.

  Let $\Gamma'$ be the subgroup of $\Gamma$ generated by all products
  \begin{equation*}
    g_{i_{1}} g_{i_{2}} \dots g_{i_{N}}.
  \end{equation*}

  Then $\Gamma'$ is a normal subgroup of $\Gamma$.
\end{lemma}

\begin{proof}
  We have to show that for all choices of generators $g = g_{i_{0}},
  g_{i_{1}}, g_{i_{2}}, 
  \dots , g_{i_{N}}$ of $\Gamma$, the
  conjugate $g^{-1} (g_{i_{1}} g_{i_{2}} \dots g_{i_{N}}) g$ lies in
  $\Gamma'$.

  We have
  \begin{equation*}
    g^{-1} (g_{i_{1}} g_{i_{2}} \dots g_{i_{N}}) g
    =
    (g^{N})^{-1} (g^{N-1} g_{i_{1}}) (g_{i_{2}} \dots g_{i_{N}} g)
    \in
    \Gamma'.
  \end{equation*}
\end{proof}

Clearly our result for $\Gamma$ has no immediate implication about the
size of the smaller group $\mathcal{G}$ of GOST.

\section{Primitivity}
\label{sec:pri}

We recall a couple of basic  properties of imprimitive groups. Let $G$
be a finite group acting transitively on a set $V$.
\begin{lemma}\label{lemma:block}
  A block  (of imprimitivity)  is of  the form $v  H$, for some $v  \in V$,
  and some proper subgroup  $H$ of $G$ which  properly  contains  the
  stabiliser of $v$  in $G$.
\end{lemma}

\begin{lemma}
  If $T$ is a transitive subgroup of $G$, then a block
  for $G$ is also a block for $T$.
\end{lemma}

In our case, $\mathcal{T}$ is a transitive subgroup of $\Gamma$. We
first record a trivial observation, which is an immediate consequence
of the fact that the map $v \mapsto \rho_{v}$ is an isomorphism $(V,
\boxplus) \to \mathcal{T}$.
\begin{lemma}\label{lemma:trivobs}
  The subgroups of $\mathcal{T}$ are of the form 
  \begin{equation*}\label{eq:a-subgroup}
    \Set{ \rho_{u} : u \in U },
  \end{equation*}
  where $U$ is a subgroup of $(V, \boxplus)$.
\end{lemma}
We obtain
\begin{lemma}\label{rem1}
  If $\Gamma$ acting on $V$ has  a block system, then this consists of
  the cosets  of a $\boxplus$-subgroup of  $V$, that is, it  is of the
  form
  \begin{equation*}
    \Set{
      W \boxplus v
      :
      v \in V
      }
  \end{equation*}
  where $W$ is a non-trivial, proper subgroup of $(V, \boxplus)$.
\end{lemma} 

According to Lemma~\ref{rem1}, to prove the primitivity of $\Gamma$ we
have to show that no subgroup of $(V, \boxplus)$ is a block.
 Goursat has characterized~\cite[Sections 11--12]{goursat} the
subgroups of the direct product  of two groups in terms of
suitable sections of the direct factors. (See
also~\cite{Petrillo}.) 
\begin{theo}[Goursat's Lemma]
  \label{theorem:Goursat}
  Let $(G_1,\boxplus)$  and $(G_2,\boxplus)$ be two  groups. There
  exists  a bijection between  
  \begin{enumerate}
  \item 
    the set  of all subgroups  of the 
    direct  product  $G_1\times   G_2$,  and  
  \item 
    the  set   of  all  triples
    $(A/B,C/D,\psi )$, where 
    \begin{itemize}
    \item $A$ is a subgroup of $G_{1}$,
    \item $C$ is a subgroup of $G_{2}$,
    \item $B$ is a normal subgroup of $A$,
    \item $D$ is a normal subgroup of $C$, and
    \item $\psi: A/B\to C/D$ is a group isomorphism.
    \end{itemize}
  \end{enumerate}

  In this bijection, each subgroup of $G_1\times G_2$ can be uniquely
  written as
  \begin{equation}\label{eq:Uphi}
    U_{\psi}= \Set{
      (a,c) \in A \times C 
      :
      (a \boxplus B) \psi =c \boxplus D
      }.
  \end{equation}
\end{theo}
Let  us consider  the case  when $G_{1}  = G_{2}  = \Z_{2^{n}}$,  with
operation $\boxplus$. Then  $A = \Span{2^{s}}$ and  $C = \Span{2^{t}}$
for  some  $s,  t$, with $0 \le s, t \le n$.
Assume first that
$s \le  t$. Therefore
there is an odd integer $z\geq 1$ such that
\begin{equation*}
  (2^{s} \boxplus B) \psi = z 2^{t} \boxplus D.
\end{equation*}
Let us  consider the endomorphism  $\phi : x  \mapsto z 2^{t-s}  x$ of
$\Z_{2^{n}}$.  Since $2^{s}  \phi  =  z 2^{t}$,  we  have that  $\phi$
induces $\psi$, that is, for $a \in A$
\begin{equation}\label{eq:psivsphi}
  (a \boxplus B) \psi = a \phi \boxplus D.
\end{equation}
If $t \le s$, we have similarly that for the endomorphism $\phi : x
\mapsto z 2^{s - t}  x$ of
$\Z_{2^{n}}$ one has
\begin{equation}\label{eq:psivsphi2}
  (c \boxplus D) \psi^{-1} = c \phi \boxplus B.
\end{equation}
We claim
\begin{lemma}\label{lemma:psiforphi}
  In the above notation, we have 
 \begin{equation}\label{eq:upsi}
    U_{\psi}
    =
    \Set{
      (a, a \phi \boxplus d)
      :
      a \in A, d \in D
      }\quad \text{when $s \le t$,}
  \end{equation}
     
 \begin{equation}\label{eq:upsi2}
    U_{\psi}
    =
    \Set{
      (c \phi \boxplus b, c)
      :
      c \in C, b \in B
      } \quad\text{when $t \le s$.}
  \end{equation}
 
\end{lemma}

\begin{proof}
  We will prove only the first equality,  the proof of the other
  being analogous. 

  Note first that the  right-hand side of~\eqref{eq:upsi} is contained
  in $U_{\psi}$, since for $a \in A$ and $d \in D$ we have
  \begin{equation*}
    (a \boxplus B) \psi = a\varphi \boxplus D =  a\varphi\boxplus d
    \boxplus D, 
  \end{equation*}
  that is, $(a, a\varphi\boxplus d) \in U_{\psi}$. 

  We now prove that $U_{\psi}$ is contained in the right-hand side
  of~\eqref{eq:upsi}. If $(a, c) \in U_{\psi}$ we have,
  using~\eqref{eq:psivsphi} 
  \begin{equation*}
    a \phi \boxplus D
    =
    (a \boxplus B) \psi
    =
    c \boxplus D,
  \end{equation*}
  so that $c = a \phi \boxplus d$ for some $d \in D$.
\end{proof}

We now show that no subgroup $U$ of $\Z_{2^{n}}\times \Z_{2^{n}}$ is a
block. By Lemma \ref{rem1}, we have to prove the following
\begin{lemma}\label{lemma:forwreath}
  There is no nontrivial, proper $\boxplus$-subgroup $U$ of $V$, and
  $(v_{1}, v_{2}) \in V$ such
  that
  \begin{equation}\label{eq:sigmacoset}
    U \Sigma = U \boxplus (v_{1}, v_{2}).
  \end{equation}
\end{lemma}

\begin{proof}
  By Theorem~\ref{theorem:Goursat} and 
  Lemma~\ref{lemma:psiforphi},
  there is $\phi \in \End(\Z_ {2^{n}})$ such that 
  \begin{equation}\label{eq:U}
    U = \Set{
      (a, a \phi \boxplus d)
      :
      a \in A, d \in D
    }
  \end{equation}
  for some  $A \le
  \Z_ {2^{n}}$ and $D \le A \phi$, or
  \begin{equation}\label{eq:U2}
    U = \Set{
      (c \phi \boxplus b,c)
      :
      c \in C, b \in B
    }
  \end{equation}
  for some  $C \le
  \Z_ {2^{n}}$ and $B \le C \phi$.

  Suppose first that $U$ satisfies \eqref{eq:sigmacoset} and \eqref{eq:U}.
  By the definition~\eqref{eq:Sigma}~and \eqref{eq:Sigma-acts} of
  $\Sigma$, we have 
  \begin{equation*}
    (a, a \phi \boxplus d) \Sigma
    =
    (a \phi \boxplus d, a + (a \phi \boxplus d) S).
  \end{equation*}
  Setting $a = d  = 0$, we see that $(0,0)\Sigma=(0,0S)$  so that we can
  take $v_{1} = 0$  and $v_{2} = 0 S$. We have thus  that for any $a \in
  A, d \in D$, there are $x \in A, y \in D$ such that
  \begin{equation}\label{eq:sigmadoes}
    (a \phi \boxplus d, a + (a \phi \boxplus d) S)
    =
    (x, x \phi \boxplus y \boxplus 0 S)\in U\boxplus (0,0S),
  \end{equation}
  that is, $x=a\varphi\boxplus d$ and $y\boxplus 0S=a+(a\varphi\boxplus
  d)S\boxminus (a\varphi\boxplus d)\varphi$, and so 
  \begin{equation}\label{eq:keyequality}
    a + (a \phi \boxplus d) S
    \boxminus
    (a \phi \boxplus d) \phi 
    \in
    0 S \boxplus D.
  \end{equation}
  Note that in the equation  $x=a\varphi\boxplus d$, $a$ and $x$ range
  in $A$ while $d$ ranges in $D$. Since $D \le A \phi$, we obtain that
  $A \phi  = A$,  and so  $s = t$,  and $\phi$  is an  automorphism of
  $\Z_{2^{n}}$.

  Setting $a = 0$ in~\eqref{eq:keyequality}, we see that $D S \subseteq
  0 S \boxplus D$. Since $S$ is bijective, we have $\Size{ D S } =
  \Size{D} =
  \Size{ 0 S \boxplus D}$, so that
  \begin{equation}\label{eq:DSequalsetc}
    D S = 0 S \boxplus D.
  \end{equation}

  When $D = \Z_{2^{n}}$, since $\phi$ is an automorphism of
  $\Z_{2^{n}}$, we have also $C = B = A = \Z_{2^{n}}$ in
  Theorem~\ref{theorem:Goursat}, so that $U = V$, a trivial block.

  In Subsection~\ref{subsec:DS} (see Corollary~\ref{cor:Wnotcoset}) we
  will show that for $D < \Z_{2^{n}}$, the
  identity~\eqref{eq:DSequalsetc} can only hold when $D = 
  \Set{0}$. Then in Subsection~\ref{subsec:diagonal} we deal with the
  case $D = \Set{0}$, that is,
  \begin{equation*}
    U = \Set{
      (a, a \phi)
      :
      a \in A
    }.
  \end{equation*}

  It remains to deal with case \eqref{eq:U2}. Recalling that $U \Sigma =
  U \boxplus (0, 0S)$, we argue as in the first case and deduce that for
  $c \in C, b \in B$, there are $x \in C, y \in B$ such that
  \begin{equation*}
    (c, (c \phi \boxplus b) + c S)
    =
    (x \phi \boxplus y, x  \boxplus 0 S).
  \end{equation*}
  Setting  $y =  0$, we  obtain  $C =  C \phi  $,  and so  $\phi$ is  an
  automorphism.  But  then, by  \eqref{eq:psivsphi2}, we  have $\Size{B} = 
  \Size{D}$ and  so $A=C$ and  $B=D$.  Setting $a  = c \phi  \boxplus b$
  in~\eqref{eq:U2}, we obtain
  \begin{align*}
    U 
    &= 
    \Set{
      (a, (a \boxminus b) \phi^{-1})
      :
      a \in A, b \in B
    }
    \\&=
    \Set{
      (a , a \phi^{-1} \boxminus b \phi^{-1})
      :
      a \in A, b \in B
    }
    \\&=
    \Set{
      (a , a \phi^{-1} \boxplus d)
      :
      a \in A, d \in D
    },
  \end{align*}
  so that we have reduced to the previous case.
\end{proof}

\subsection{The case $D S = 0 S \boxplus D$, with $D\ne\Set{0}$}
\label{subsec:DS}

For $v  \in \F_{2}^{n}$, we denote  by $v_{[h,k]}$ the string  of bits
consisting of the bits of $v$ from the  $h$-th bit to the $k$-th
bit. (We start counting from $0$.) For
example   if  $v=(0,1,1,0)$,   then   $v_{[1,3]}=(1,1,0)$.   For   any
$W\subseteq \F_{2}^{n}$ we denote  by $W_{[h,k]}$ the set $\{v_{[h,k]}
: v\in W\}$. 

According to Lemma~\ref{lemma:subgroups}, a subgroup  $D$ of
$\mathbb{Z}_{2^{n}}$  is of the form $\Span{2^{q}}$, for some $0 \le q <
  n$.  So the representation of  each element of
$D = \Span{2^{q}}$  as  an  element  of
$\F_{2}^{n} = \F_{2}^{q} \times \F_{2}^{n-q}$  is  of  the  form 
$0_{[0,q-1]} \mid\mid      d_{[q,n-1]}$       with      $d_{[q,n-1]}\in
\F_2^{n-q}$. Recall that $\F_{2}^{n}=\F_{2}^{m} \mid\mid \cdots
\mid\mid \F_{2}^{m}$. 

We   shall   use  the   following   compact
notation:
\begin{enumerate}
\item 
\psscalebox{1.0 1.0} 
{
\begin{pspicture}(0,0)(5.2,0.6)
\psframe[linecolor=black, linewidth=0.04, dimen=outer](5.2,0.6)(3.6,-0.6)
\rput[bl](0.0,0.0){A white box}
\end{pspicture}
}\\[2mm]

\noindent denotes a subset of $\F_{2}^{m}$ of cardinality 1;\\

\item 

  \psscalebox{1.0 1.0} 
{
\begin{pspicture}(0,0)(5.2,0.6)
\psframe[linecolor=black, linewidth=0.04, fillstyle=vlines, hatchwidth=0.02, hatchangle=0.0, hatchsep=0.0612, dimen=outer](5.2,0.6)(3.6,-0.6)
\rput[bl](0.0,0.0){a ruled box}
\end{pspicture}
}\\[2mm]

\noindent denotes a subset of $\F_{2}^{m}$ of cardinality $1<t<2^{m}$;\\

\item 

\psscalebox{1.0 1.0}
{
\begin{pspicture}(0,0)(5.2,0.6)
\definecolor{colour0}{rgb}{0.2,0.2,0.2}
\psframe[linecolor=black, linewidth=0.04, fillstyle=solid,fillcolor=black, dimen=outer](5.2,0.6)(3.6,-0.6)
\rput[bl](0.0,0.0){a black box}
\end{pspicture}
}\\[2mm]

\noindent denotes the full set $\F_{2}^{m}$.
\end{enumerate}
We will say that a box has white, ruled or black \emph{type}.

We will also speak of
\begin{enumerate}
\item[(4)] 

     \begin{pspicture}(-0.15,0)(5.2,0.6)
       \pspolygon(3.6,-0.6)(3.6,0.6)(5.2,0.6)(5.2,-0.6)(3.6,-0.6)
       \rput(4.4,0){\textbf{\large ?}}
     \rput[bl](0.0,0.0){a riddle box}
     \end{pspicture}\\[2mm]

\noindent which is any of the above.
\end{enumerate}

\begin{definition}\label{def:fortype}
  Let $D$ be a subset of
  \begin{equation*}
    \F_{2}^{n} 
    =
    V_{1} \times V_{2} \times \dots \times V_{\delta},
  \end{equation*}
  where each subspace $V_{i}$ has dimension $m$.
  We shall say that $D$ \emph{has a
    type} if 
  \begin{equation*}
    D 
    =
    (D \cap V_{1}) \times (D \cap V_{2}) \times 
    \dots 
    \times (D \cap V_{\delta}). 
  \end{equation*}
  If $D$  has a  type, the \emph{type}  of $D$ will  be a  sequence of
  $\delta$  white,  ruled  or  black   boxes,  where  the  $i$-th  box
  represents the set $D \cap V_{i}$.
\end{definition}

\begin{remark}\label{rem:typesofsubgroups}
  A subgroup $D = \Span{2^{q}}$ of
  $\mathbb{Z}_{2^n}$ has one of the following two types.

  \begin{enumerate}
  \item When $q \equiv 0 \pmod{m}$, the subgroup has
    type:

  \psscalebox{1.0 1.0} 
             
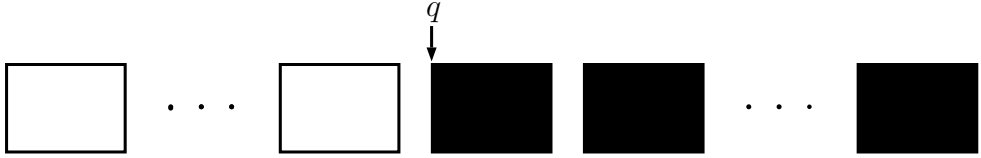
{
               \begin{pspicture}(0,-0.5)(12.8,1.5)
                 \psframe[linecolor=black, linewidth=0.04, fillstyle=solid, dimen=outer](5.2,0.6)(3.6,-0.6)
                 \psframe[linecolor=black, linewidth=0.04, fillstyle=solid,fillcolor=black, dimen=outer](12.8,0.6)(11.2,-0.6)
                 \psframe[linecolor=black, linewidth=0.04, fillstyle=solid,fillcolor=black, dimen=outer](9.2,0.6)(7.6,-0.6)
                 \psframe[linecolor=black, linewidth=0.04, fillstyle=solid,fillcolor=black, dimen=outer](7.2,0.6)(5.6,-0.6)
                 \psframe[linecolor=black, linewidth=0.04, fillstyle=solid, dimen=outer](1.6,0.6)(0.0,-0.6)
                 \psdots[linecolor=black, dotsize=0.06531812](2.18,0.0)
                 \psdots[linecolor=black, dotsize=0.06531812](2.18,0.02)
                 \psdots[linecolor=black, dotsize=0.06531812](2.98,0.02)
                 \psdots[linecolor=black, dotsize=0.06531812](2.98,0.02)
                 \psdots[linecolor=black, dotsize=0.06531812](2.58,0.02)
                 \psdots[linecolor=black, dotsize=0.06531812](2.58,0.02)
                 \psdots[linecolor=black, dotsize=0.066](9.78,0.02)
                 \psdots[linecolor=black, dotsize=0.066](10.18,0.02)
                 \psdots[linecolor=black, dotsize=0.06](10.58,0.02)
                 \rput[bl](5.54,1.14){$q$}
                 \psline[linecolor=black, linewidth=0.04, arrowsize=0.05291666666666667cm 2.0,arrowlength=1.4,arrowinset=0.0]{->}(5.61,1.09)(5.61,0.62)
               \end{pspicture}
             }\\[2mm]

\noindent Here there are no ruled  boxes, and the $q$-th bit occurs as
the first bit of a black box.  Note that there are no white boxes when
$q = 0$  (the subgroup is the full group  $\Z_{2^{n}}$), and there are
no black boxes when $q = 2^{n}$ (the subgroup is $\Set{0}$).

\item When $q \not\equiv 0 \pmod{m}$, there is 
  a ruled box:
  
  \psscalebox{1.0 1.0} 
             
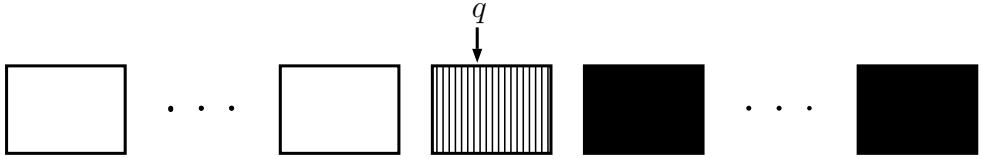
{
               \begin{pspicture}(0,-1)(12.8,1.2)
                 \definecolor{colour0}{rgb}{0.2,0.2,0.2}
                 \psframe[linecolor=black, linewidth=0.04, fillstyle=solid, dimen=outer](5.2,0.22)(3.6,-0.98)
                 \psframe[linecolor=black, linewidth=0.04, fillstyle=solid,fillcolor=black, dimen=outer](12.8,0.22)(11.2,-0.98)
                 \psframe[linecolor=black, linewidth=0.04, fillstyle=solid,fillcolor=black, dimen=outer](9.2,0.22)(7.6,-0.98)
                 \psframe[linecolor=black, linewidth=0.04, fillstyle=vlines, hatchwidth=0.02, hatchangle=0.0, hatchsep=0.0612, dimen=outer](7.2,0.22)(5.6,-0.98)
                 \psframe[linecolor=black, linewidth=0.04, fillstyle=solid, dimen=outer](1.6,0.22)(0.0,-0.98)
                 \psdots[linecolor=black, dotsize=0.06531812](2.18,-0.38)
                 \psdots[linecolor=black, dotsize=0.06531812](2.18,-0.36)
                 \psdots[linecolor=black, dotsize=0.06531812](2.98,-0.36)
                 \psdots[linecolor=black, dotsize=0.06531812](2.98,-0.36)
                 \psdots[linecolor=black, dotsize=0.06531812](2.58,-0.36)
                 \psdots[linecolor=black, dotsize=0.06531812](2.58,-0.36)
                 \psdots[linecolor=black, dotsize=0.066](9.78,-0.36)
                 \psdots[linecolor=black, dotsize=0.066](10.18,-0.36)
                 \psdots[linecolor=black, dotsize=0.06](10.58,-0.36)
                 \rput[bl](6.14,0.76){$q$}
                 \psline[linecolor=black, linewidth=0.04, arrowsize=0.05291666666666667cm 2.0,arrowlength=1.4,arrowinset=0.0]{->}(6.21,0.71)(6.21,0.24)
               \end{pspicture}
             }\\[2mm]
             where the $q$-th bit is inside the ruled box. 
  \end{enumerate}
\end{remark}

\begin{definition}
  A subgroup of $\Z_{2^n}$ of the first  type of
  Remark~\ref{rem:typesofsubgroups} will be called a \emph{whole}
  subgroup.
\end{definition}

In the next Lemma we consider the behaviour of the bitwise sum with
respect to types. 
\begin{lemma}\label{lemma:subset}
  If  $D$  is  a  subset  of $\Z_{2^{n}}$  having  a  type  and
  $v \in \Z_{2^{n}}$, then $D$ and $v+D$ have the same type.
\end{lemma}
\begin{proof}
  Since $D$ has a type, $D=D_{1}\times\cdots \times D_{\delta}$, where
  $D_{i}=D\cap  V_{i}$  for each  $i\in\{1,\ldots,\delta\}$.   Writing
  $v=(v_{1},\cdots, v_{\delta})$, clearly we have
  \begin{equation*}
    D+v=(D_{1}+v_{1})\times\cdots\times(D_{\delta}+v_{\delta})
  \end{equation*}   
  and so $D+v$ has a  type. Since $\Size{D_{i}} = \Size{D_{i}+v_{i}}$,
  the two types coincide.
\end{proof}

The behaviour of the modular sum $\boxplus$ with respect to types is
more complex and can be described easily only for subgroups, as in the
following lemma. 
\begin{lemma}\label{lemma:Wform}
  If  $D$  is  a  subgroup  of $\mathbb{Z}_{2^n}$ and
  $v\in\mathbb{Z}_{2^n}$,  then $D$  and 
  $v\boxplus D$ have the same type.
\end{lemma}

\begin{proof}
  The binary representation  of an  element $d$ of  $D =
  \Span{2^{q}}$  has the  form 
  \begin{equation*}
      d = 0_{[0,q-1]} \mid\mid  d_{[q,n-1]},
  \end{equation*}
  where $0_{[0,q-1]}$  is a zero vector of length $q$.
  Write $v = v_{[0,q-1]} \mid\mid v_{[q,n-1]}$. Then an
  element $v\boxplus d$ of $v\boxplus D$ can be written as
  \begin{align*}
    v \boxplus d
    &=
    (v_{[0,q-1]} \mid\mid v_{[q,n-1]}) 
    \boxplus
    (0_{[0,q-1]} \mid\mid d_{[q,n-1]})
    \\&=
    (v_{[0,q-1]} \boxplus 0_{[0,q-1]}) \mid\mid (v_{[q,n-1]} \boxplus d_{[q,n-1]})
    \\&=
    v_{[0,q-1]} \mid\mid (v_{[q,n-1]}\boxplus d_{[q,n-1]}).
  \end{align*}
  As $d_{[q,n-1]}$ ranges in $\F_{2}^{n-q}$, so does
  $v_{[t,n-1]}\boxplus d_{[q,n-1]}$. Therefore $D$ and $v \boxplus D$
  have the same type.
\end{proof}

Clearly a bricklayer transformations  will map  any set having a
type to another set having the same type, since each S-box is
a bijection.

\begin{lemma}\label{lemma:Wgammaform}
  If $D$ is a subgroup  of $\mathbb{Z}_{2^n}$, then $D$, $D\gamma$ and
  $0\gamma\boxplus  D$   have  the  same  type,   for  any  bricklayer
  transformation $\gamma \in \Sym(V)$.  

  Moreover, if $D$ is
  whole, then $D\gamma=0\gamma\boxplus  D$.
\end{lemma}

\begin{proof}
  Clearly   $D$   and   $D\gamma$   share  the   same   type   and   by
  Lemma~\ref{lemma:Wform}, this is the same type as $0\gamma\boxplus
  D$.
  
  If $D$
  is a whole subgroup, then
  \begin{equation*}
    D = 0_{[0,m-1]} \mid\mid \cdots  \mid\mid 
    0_{[(l-1)m,lm-1]}
    \mid\mid D_{l} \mid\mid \cdots  \mid\mid  
    D_{\delta}
  \end{equation*}
  for some $l\leq \delta$, and thus 
  \begin{equation*}
    D\gamma=0\gamma_{1} \mid\mid \cdots
     \mid\mid 0\gamma_{l-1} \mid\mid D_{l}\gamma_{l} 
     \mid\mid \cdots  \mid\mid 
    D_{\delta}\gamma_{\delta}.
  \end{equation*}
  Since $D_{i}=\F_{2^{m}}$ for any $i\in\{l,\ldots,\delta\}$,
  $D\gamma=0\gamma\boxplus D$.
\end{proof}

\begin{lemma}\label{lemma:notWform}
  If $D$ is a proper, nontrivial subgroup of $\mathbb{Z}_{2^n}$, then
  $DS$ and $D$ have   different types.
\end{lemma}

\begin{proof}
  By Lemma \ref{lemma:Wgammaform} we know  that $D$ and $D\gamma$ have
  the same type.  We will now prove that an  application of $R$ changes
  the type,  which will yield the claim.
  According to Remark~\ref{rem:typesofsubgroups}, we distinguish
  the following three possibilities for the type of  $D \gamma$:\\

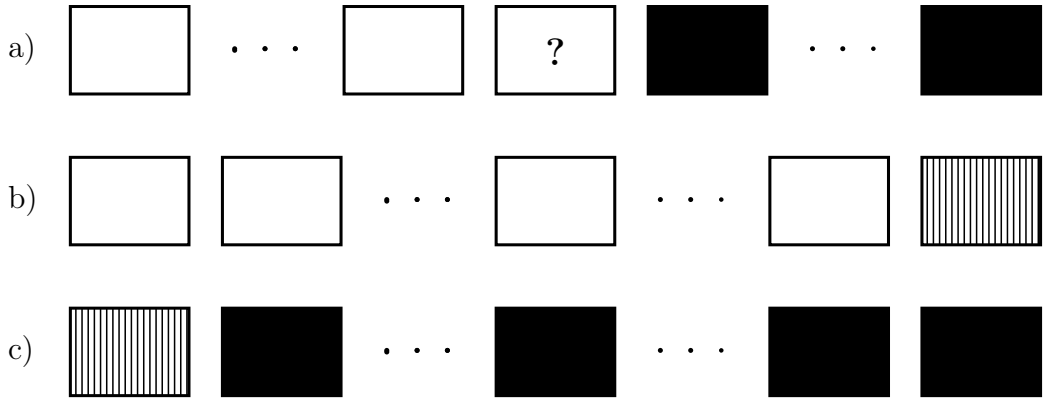
\psscalebox{1.0 1.0} 
{
\begin{pspicture}(0,-2.6)(13.6,2.6)
\psframe[linecolor=black, linewidth=0.04, fillstyle=solid, dimen=outer](6.0,2.6)(4.4,1.4)
\psframe[linecolor=black, linewidth=0.04, fillstyle=solid,fillcolor=black, dimen=outer](13.6,2.6)(12.0,1.4)
\psframe[linecolor=black, linewidth=0.04, fillstyle=solid,fillcolor=black, dimen=outer](10.0,2.6)(8.4,1.4)
\psframe[linecolor=black, linewidth=0.04, fillstyle=solid, dimen=outer](8.0,2.6)(6.4,1.4)
\psframe[linecolor=black, linewidth=0.04, fillstyle=solid, dimen=outer](2.4,2.6)(0.8,1.4)
\psdots[linecolor=black, dotsize=0.06531812](2.98,2.0)
\psdots[linecolor=black, dotsize=0.06531812](2.98,2.02)
\psdots[linecolor=black, dotsize=0.06531812](3.78,2.02)
\psdots[linecolor=black, dotsize=0.06531812](3.78,2.02)
\psdots[linecolor=black, dotsize=0.06531812](3.38,2.02)
\psdots[linecolor=black, dotsize=0.06531812](3.38,2.02)
\psdots[linecolor=black, dotsize=0.066](10.58,2.02)
\psdots[linecolor=black, dotsize=0.066](10.98,2.02)
\psdots[linecolor=black, dotsize=0.06](11.38,2.02)
\psframe[linecolor=black, linewidth=0.04, fillstyle=solid,fillcolor=black, dimen=outer](4.4,-1.4)(2.8,-2.6)
\psframe[linecolor=black, linewidth=0.04, fillstyle=solid,fillcolor=black, dimen=outer](13.6,-1.4)(12.0,-2.6)
\psframe[linecolor=black, linewidth=0.04, fillstyle=solid,fillcolor=black, dimen=outer](11.6,-1.4)(10.0,-2.6)
\psframe[linecolor=black, linewidth=0.04, fillstyle=solid,fillcolor=black, dimen=outer](8.0,-1.4)(6.4,-2.6)
\psframe[linecolor=black, linewidth=0.04, fillstyle=vlines, hatchwidth=0.02, hatchangle=0.0, hatchsep=0.0612, dimen=outer](2.4,-1.4)(0.8,-2.6)
\psdots[linecolor=black, dotsize=0.06614429](4.98,-2.0)
\psdots[linecolor=black, dotsize=0.06614429](4.98,-1.98)
\psdots[linecolor=black, dotsize=0.06614429](5.78,-1.98)
\psdots[linecolor=black, dotsize=0.06614429](5.78,-1.98)
\psdots[linecolor=black, dotsize=0.06614429](5.38,-1.98)
\psdots[linecolor=black, dotsize=0.06614429](5.38,-1.98)
\psdots[linecolor=black, dotsize=0.066](8.58,-1.98)
\psdots[linecolor=black, dotsize=0.066](8.98,-1.98)
\psdots[linecolor=black, dotsize=0.06](9.38,-1.98)
\psframe[linecolor=black, linewidth=0.04, fillstyle=solid, dimen=outer](4.4,0.6)(2.8,-0.6)
\psframe[linecolor=black, linewidth=0.04, fillstyle=vlines, hatchwidth=0.02, hatchangle=0.0, hatchsep=0.0612, dimen=outer](13.6,0.6)(12.0,-0.6)
\psframe[linecolor=black, linewidth=0.04, fillstyle=solid, dimen=outer](11.6,0.6)(10.0,-0.6)
\psframe[linecolor=black, linewidth=0.04, fillstyle=solid, dimen=outer](8.0,0.6)(6.4,-0.6)
\psframe[linecolor=black, linewidth=0.04, fillstyle=solid, dimen=outer](2.4,0.6)(0.8,-0.6)
\psdots[linecolor=black, dotsize=0.06531812](4.98,0.0)
\psdots[linecolor=black, dotsize=0.06531812](4.98,0.02)
\psdots[linecolor=black, dotsize=0.06531812](5.78,0.02)
\psdots[linecolor=black, dotsize=0.06531812](5.78,0.02)
\psdots[linecolor=black, dotsize=0.06531812](5.38,0.02)
\psdots[linecolor=black, dotsize=0.06531812](5.38,0.02)
\psdots[linecolor=black, dotsize=0.066](8.58,0.02)
\psdots[linecolor=black, dotsize=0.066](8.98,0.02)
\psdots[linecolor=black, dotsize=0.06](9.38,0.02)
\rput[bl](0.0,1.8){a)}
\rput[bl](0.0,-0.2){b)}
\rput[bl](0.0,-2.2){c)}
\rput[b](7.2,1.8){\textbf{\large ?}}
\end{pspicture}
}\\[2mm]
As in Definition \ref{def:fortype}, we count the boxes from $1$ to
$\delta$.

\begin{itemize}
\item 
Consider first case  a), when we have both black  and white boxes, and
the riddle box can be of any type.

If $r = 2 m$, then the white box preceding the riddle box is sent by $R$
onto the black box following the riddle box, a contradiction.

Similarly, if $r = (\delta - 1) m$, the first white box is sent by $R$ onto
the last black box. For later use, we regard this as $R^{-1}$ sending
the last black box onto the first white box.

Now note first that every $m$-bit box  that is contained in the stretch of
white boxes will be  white, even if it is not aligned  with one of the
bricks $V^{i}_{j}$. This is simply because all  bits in this stretch
take a single 
value  each. Similarly,  every $m$-bit  box that  is contained  in the
stretch of black boxes  will be black, even if it  is not aligned with
one of  the bricks. This is because  all bits in this  stretch take two
values  each, independent of one another.

To deal  with the intermediate  cases $2m \le r  \le (\delta -  1) m$,
start with  the case $r =  2 m$, and shift the black box next to the
riddle box right by one bit. As just noted, this will still be black,
and for $r = 2 m + 1$, the rotation $R$ will take the white box next
to the riddle box onto the shifted black box, a contradiction. 

We keep shifting the
black box to the right one bit at a time, until we hit the rightmost
black box. In this way we will have covered all rotations $R$, for $2
m \le r \le \theta m$, where $\delta - \theta + 1$ is the position of the
riddle box.
 
To cover the remaining rotations, start with the last black box, which
for $r =  \theta m$ is taken  by the left rotation $R^{-1}$  onto the white
box adjacent to the riddle box. Shift the latter white box left by one
bit. By the  remark above, this will still be  white, and the left rotation
$R^{-1}$ by $r =  \theta m + 1$ bits will take the last black box onto it,
a contradiction.

Shifting  bit by  bit the  white box  to the  left, until  it overlaps
completely the first white box, we see  that for $2m \le r \le (\delta
- 1) m$, one or both of the following possibilities will have occurred.
\begin{enumerate}
\item The rotation $R$ sends a white box onto a black box, or
  over two adjacent black boxes. Since in a white box all bits take a
  single value, while in a black box each bit takes two values,
  independent of one another, this is a contradiction.
\item The left rotation $R^{-1}$ sends a black box onto a white box, or
  over two adjacent white boxes. This is again a contradiction.
\end{enumerate}

If $m \le r < 2 m$, then $R$ sends the last black box, either onto the
first white box, or in any case to overlap the first white box in the
last $2 m - r > 0$ bits of the latter. Once more, this is a contradiction.

\item Consider now case  b). Here we do  not have black boxes  and the ruled
box is the rightmost  one, at position $\delta$.  
Under the rotation
to the right by $r$ bits,
the ruled box is moved onto a white box, or comes to overlap two
adjacent  white boxes. This implies that all bits of the ruled box
take a single value each, so that the ruled box is a singleton, that
is, it is also white, a
contradiction. 

\item Finally, in case c)  we do not have white boxes, and  the ruled box is
the leftmost  one. Applying a rotation  to the right by  $r$ bits, the
ruled box is moved onto a black  box, or comes to overlap two adjacent
black  boxes.  Since concatenation  of  boxes  means concatenation  of
strings, and in a black box  each bit takes two values, independent of
one another, this would make the ruled box black, a contradiction.
\end{itemize}
\end{proof}

\begin{cor}\label{cor:Wnotcoset}
  If $D  \ne \Set{ 0 }$  is a  subgroup of  $\mathbb{Z}_{2^n}$, then
    $DS\ne 0S\boxplus D$.
\end{cor}

\begin{proof}
  It    follows   from    Lemma    \ref{lemma:notWform}   and    Lemma
  \ref{lemma:Wform}.
\end{proof}

\subsection{The diagonal case $D = \Set{0}$}
\label{subsec:diagonal}

Here we deal with the case when a subgroup of the form
\begin{equation*}
  U = \Set{
    (a, a \phi)
    :      
    a \in A 
  },
\end{equation*}
for some $0 \ne A \le \mathbb{Z}_{2^n}$ and  $\phi \in \Aut(\Z_{2^n})$,
is a block. 
Since 
\begin{equation*}
  (a,a \phi) \Sigma = (a \phi, a + a \phi S),
\end{equation*}
as in the discussion following Lemma~\ref{lemma:forwreath} we have
\begin{equation*}
  U \Sigma = U \boxplus (0, 0S).
\end{equation*}
Therefore for each $a \in A$ there is $x \in A$ such that
\begin{equation*}
  (a \phi, a + a \phi S)
  =
  (x, x \phi \boxplus 0 S),
\end{equation*}
so that $x=a\varphi$, and substituting
\begin{equation}\label{eq:willfail}
  a \phi^{2}
  =
  (a + a \phi S)  \boxminus 0 S.  
\end{equation}
Since $\phi$ is an automorphism, we have $2^{n-1} \phi = 2^{n-1}$. Now
for any $y$ it is easy to see that   
\begin{equation*}
  y + 2^{n-1} = y \boxplus 2^{n-1},
\end{equation*}
as in both cases we are just changing the most significant bit of
$y$. Therefore, setting $a = 2^{n-1} \in A$ in~\eqref{eq:willfail}, we
obtain 
\begin{equation*}
  2^{n-1} = 2^{n-1} \boxplus 2^{n-1} S \boxminus 0 S,
\end{equation*}
or in other words
\begin{equation*}
2^{n-1} S = 0 S,
\end{equation*} 
contradicting the fact that $S$ is a bijection.

\section{O'Nan-Scott}
\label{sec:ON-S}

We have  shown in the previous  section that the subgroup  $\Gamma$ of
$\Sym(V)$ is  primitive.  We  may thus prove      Theorem
\ref{theo:sym} by appealing to   the O'Nan-Scott
classification   of   primitive   groups~\cite{LPS}.
However, since by~\eqref{eq:gamma}  $\Gamma$ contains the
group  $\mathcal{T}$ of  translations, which  is an  abelian subgroup
acting regularly on $V$, we  are    able  to  appeal   to a
particular case of the O'Nan-Scott classification, obtained by
Li~\cite[Theorem 
  1.1]{Li},  which describes the primitive  groups
containing  an abelian 
regular subgroup. In  the particular case when $\Gamma$ acts on a set
whose order is a
power of $2$, Li's result can be stated as follows.
\begin{theo}[\cite{Li}, Theorem 1.1]\label{Li}
  Let $\Gamma$ be  a primitive group  acting on a set $V$ of
  cardinality $2^{b}$, with $b > 1$. Suppose
  $\Gamma$ contains a regular abelian subgroup $T$.

  Then $\Gamma$ is one of the following.
  \begin{enumerate}
  \item Affine, $\Gamma \leq  \AGL(b,2)$.
  \item Wreath product, that is
    $$ \Gamma \cong (K_1\times\cdots\times K_l).O.P,
    $$  with  $2^{b}  =  c^{l}$  for  some $c$  and  $l  >  1$.   Here
    $T=T_1\times   \cdots  \times   T_l$,  with   $T_i\leq  K_i$   and
    $\Size{T_{i}} = c$ for each $i$, $K_1\cong\ldots\cong K_l$, $O\leq
    \Out(K_1)\times\cdots\times\Out(K_l)$,         $P$
    permutes transitively the  $K_i$, and either $K_{i}  = \Sym(c)$ or
    $K_{i} = \Alt(c)$.
  \item  Almost  simple, i.e.,  $K\leq  \Gamma\leq  \Aut(K)$ for  a
    nonabelian simple group $K$.
  \end{enumerate}
\end{theo}
Here the notation  $S.T$ denotes an extension of the  group $S$ by the
group $T$.

Case (2) is the case of the (wreath product in) \emph{product
  action}. In dealing with this, we will be supplementing Li's
statement with the information from~\cite{LPS}.

In the next three subsections we will examine the three cases of
Theorem~\ref{Li}, and show that only the almost simple case can hold,
with $\Gamma = \Alt(V)$.  

Recall that  in our case $\Size{V}  = 2^b$, with $b=2n$,  $n=\delta m$
with $\delta \geq 4$ and $m\geq  2$. These conditions imply that $b\geq
16$ and $n\geq 8$.

\subsection{The affine case}

Suppose  case (1)  of  Theorem~\ref{Li} holds,  that  is, $\Gamma  \le
\AGL(2 n, 2)$.  Then $\AGL(2 n, 2)$ should contain the cyclic subgroup
$\Z_{2^{n}}$.

It is well known that  if $p$ is a  prime, then the
exponent of  the $p$-Sylow subgroup of  $\GL(2 n, p)$  is the smallest
power $p^{k}$ such that $p^{k} \ge 2 n$.
In our case the exponent of the $2$-Sylow subgroup of $\GL(2 n, 2)$ is
the smallest power $2^{k} \ge  2 n$, so that $k \ge \log_{2}(n) + 1$, and
\begin{equation*}
  k 
  = 
  \lceil \log_{2}(n)  + 1 \rceil 
  = 
  \lceil \log_{2}(n) \rceil  + 1.
\end{equation*}
Since $\AGL(2 n, 2)$ is the extension of an elementary abelian group
by $\GL(2 n, 2)$,
the exponent of the $2$-Sylow subgroup of $\AGL(2 n, 2)$ can only
increase by
a factor of two with respect to that of  $\GL(2 n, 2)$. Therefore if there
is an element of order $2^{n}$ in 
$\AGL(2 n, 2)$, then
\begin{equation*}
  \lceil \log_{2}(n) \rceil + 2 \ge n,
\end{equation*}
which fails for $n > 5$. (Recall that we have $n \ge 8$.)

\subsection{The wreath product case}
\label{sec:wreath}

This is  case III(b)  (wreath product in product action)
of~\cite{LPS}.  
Therefore
\begin{equation*}
  V = W_{1} \times \dots \times W_{l},
\end{equation*}
with $K_{i}$ acting  transitively on the subsets $W_{i}$,  each of the
latter having  cardinality $c >  1$.  Since  $T_{i}$ is a  subgroup of
order  $c$  of  the $\boxplus$-translation  group  $\mathcal{T}  \cong
\Z_{2^{n}} \times  \Z_{2^{n}}$, and $\mathcal{T} =  T_{1} \times \dots
\times T_{l}$, it follows that $l =  2$, $c = 2^{n}$, and $T_{i} \cong
\Z_{2^{n}}$.   By  Lemma~\ref{lemma:trivobs}, $T_{i}  =  \Set{\rho_{u}
  : u \in  U_{i}}$, where the $U_{i}$ are subgroups  of $V$.  Since
$\mathcal{T}$ acts  regularly on  $V$, we  have $W_{i} =  0 K_{i}  = 0
T_{i} = U_{i}$ , so that the $W_{i}$ are subgroups of $V$.

Since $\Gamma = \Span{ \mathcal{T}, \Sigma }$, $\mathcal{T}$ is
contained in the normal subgroup $K_{1}
\times K_{2}$, and $P$ permutes the $K_{i}$ by conjugation, it follows
that
\begin{equation*}
  \Sigma^{-1} K_{1} \Sigma = K_{2}.
\end{equation*}

We have thus
\begin{equation}\label{eq:thefinalcut}
  W_{1} \Sigma 
  =
  0 K_{1} \Sigma
  =
  0 \Sigma K_{2}
  =
  0 \Sigma T_{2} 
  =
  (0, 0 S) \boxplus W_{2}.
\end{equation}
(Here and in the following, recall~\eqref{eq:Sigma}
and~\eqref{eq:Sigma-acts}.) 

We now prove that~\eqref{eq:thefinalcut} cannot hold, with
arguments similar to those 
of Section~\ref{subsec:diagonal}.  

We  appeal once  again to  Goursat's Lemma  to describe  the subgroups
$W_{1}, W_{2}$ of $V = V^{1} \times V^{2}$. Note that, in the notation
of  Theorem~\ref{theorem:Goursat},  the  subgroup  $U_{\psi}$  of  the
direct product  contains $B \times  D$. Since $W_{1}  \cong \Z_{2^{n}}
\cong W_{2}$ are  indecomposable, one of $B$ and $D$  must be trivial.
In~\eqref{eq:upsi} of Lemma~\ref{lemma:psiforphi}, $D$ is the image of
$B$ under an endomorphism, and in~\eqref{eq:upsi2} $B$ is the image of
$D$  under  an  endomorphism.  It  follows that  in  the  notation  of
Lemma~\ref{lemma:psiforphi} $W_{1}, W_{2}$ are of one of the two forms
\begin{equation*}
  \Set{
    (x, x \sigma)
    :
    x \in \Z_{2^{n}}
  },
  \qquad
  \Set{
    (y \tau, y)
    :
    y \in \Z_{2^{n}}
  },
\end{equation*}
where $\sigma, \tau \in \End(\Z_{2^{n}})$. There are four
cases to consider.

The first case is 
\begin{equation*}
  W_{1} = \Set{
    (x, x \sigma)
    :
    x \in \Z_{2^{n}}
  },
  \quad
  W_{2} = \Set{
    (y, y \tau)
    :
    y \in \Z_{2^{n}}
  },
\end{equation*}
for $\sigma, \tau \in \End(\Z_{2^{n}})$. In this case
\eqref{eq:thefinalcut} states that for each
$y \in \Z_{2^{n}}$ there is a unique $x \in \Z_{2^{n}}$ such that
\begin{equation*}
  (x \sigma, x + x \sigma S) = (y, y \tau \boxplus 0 S).
\end{equation*}
Therefore $y = x \sigma$, and $\sigma \in \Aut(\Z_{2^{n}})$. Set $x =
2^{n-1}$. We get
\begin{equation*}
  2^{n-1} + 2^{n-1} S = 2^{n-1} \tau \boxplus 0 S.
\end{equation*}
If $\tau$ is
also an automorphism, we get $2^{n-1} S = 0S$, a contradiction to the
fact that $S$ is bijective. If $\tau$ is a \emph{proper}
endomorphism, that is, an endomorphism which is not an automorphism,
we get 
\begin{equation}\label{eq:ohno}
  2^{n-1} + 2^{n-1} S  = 0 S.
\end{equation}
Regarding this as an identity in $V^{0}$, it states that $0 S$ and
$2^{n-1} S$ differ only in the last bit. Clearly $0$ and $2^{n-1}$
differ only in the last bit, so that $0 \gamma$ and $2^{n-1} \gamma$
differ only in their component in $V^{0}_{\delta}$. But then, once one
applies the right rotation $R$ by $r$ bits, with 
$m \le r \le (\delta- 1) m$, we have
that the components in $V^{0}_{\delta}$ of 
$0 S = 0 \gamma R$ and $2^{n-1} S = 2^{n-1} \gamma R$   coincide,
contradicting~\eqref{eq:ohno}. 

The second case is
\begin{equation*}
  W_{1} = \Set{
    (x, x \sigma)
    :
    x \in \Z_{2^{n}}
  },
  \quad
  W_{2} = \Set{
    (y \tau, y)
    :
    y \in \Z_{2^{n}}
  },
\end{equation*}
for $\sigma, \tau \in \End(\Z_{2^{n}})$. We thus have that for each
$x \in \Z_{2^{n}}$ there is a unique $y \in \Z_{2^{n}}$ such that
\begin{equation*}
  (x \sigma, x + x \sigma S) = (y  \tau, y \boxplus 0 S).
\end{equation*}
Setting $x = 0$, we see that $\tau$ is an automorphism, and similarly
$\sigma$ is an automorphism. Setting $x = 2^{n-1}$, we have also $y =
2^{n-1}$, so that we get once more 
\begin{equation*}
  2^{n-1} S = 0 S.
\end{equation*}

The third case is
\begin{equation*}
  W_{1} = \Set{
    (x \sigma, x)
    :
    x \in \Z_{2^{n}}
  },
  \quad
  W_{2} = \Set{
    (y, y  \tau)
    :
    y \in \Z_{2^{n}}
  },
\end{equation*}
for $\sigma, \tau \in \End(\Z_{2^{n}})$. Thus we have that for each $x
\in \Z_{2^{n}}$, there is a unique $y \in \Z_{2^{n-1}}$ such
that
\begin{equation*}
  (x, x \sigma  + x S) = (y, y \tau \boxplus 0 S).
\end{equation*}
Therefore $x = y$, and for each
$y \in \Z_{2^{n}}$ we have
\begin{equation*}
  y \sigma + y S = y \tau \boxplus 0 S.
\end{equation*}
If   $\sigma,   \tau$  are   both   automorphisms,   or  both   proper
endomorphisms, setting $y  = 2^{n-1}$ we get once more  $2^{n-1} S = 0
S$, a contradiction. If one of  $\sigma, \tau$ is an automorphism, and
the other is a proper endomorphism,  then setting $y = 2^{n-1}$ we get
as above
\begin{equation*}
  2^{n-1} S + 2^{n-1} = 0 S,
\end{equation*}
a contradiction.

The fourth case is
\begin{equation*}
  W_{1} = \Set{
    (x \sigma, x)
    :
    x \in \Z_{2^{n}}
  },
  \quad
  W_{2} = \Set{
    (y \tau, y)
    :
    y \in \Z_{2^{n}}
  },
\end{equation*}
for $\sigma, \tau \in \End(\Z_{2^{n}})$.  We thus have that for each
$x \in \Z_{2^{n}}$ there is a unique $y \in \Z_{2^{n}}$ such that
\begin{equation*}
  (x, x \sigma + x S) = (y  \tau, y \boxplus 0 S).
\end{equation*}
It follows that $\tau$ is an  automorphism, and $y = x \tau^{-1}$. Thus for
each $x \in \Z_{2^{n}}$ one has
\begin{equation*}
  x \sigma + x S = x \tau^{-1} \boxplus 0 S,
\end{equation*}
so this case reduces to the previous one.

\subsection{The almost simple case}
\label{subsec:almostsimple}
 
In the almost simple case (3)~of  Theorem~\ref{Li}, note that $K$ is a
transitive  subgroup   of  the   primitive  group  $\Gamma$,   so  the
intersection  of a  one-point stabiliser  in  $\Gamma$ with  $K$ is  a
proper  subgroup  of $K$  of  index  $2^{b}$,  with  $b \ge  16$.   By
Theorem~1  and Section~(3.3)  in~\cite{CGC-alg-art-guralnick83}, there
are two possibilities for $K$.

The first possibility is for $K$ to be   the
group $\PSL_\alpha(\beta)$, where in our case
\begin{enumerate}
\item[(i)] $(\beta^\alpha-1)/(\beta-1)= 2^{b}$;
\item[(ii)] $\beta$ is a power $\pi^e$  of a prime $\pi$;
\item[(iii)] $\alpha$ is a prime such that if $\alpha>2$ then  $\pi \equiv
    1 \pmod{\alpha}$.
\end{enumerate}
(i) implies that $\beta$, and thus $\pi$, are odd.
Hence
\begin{equation*}
  2^{b} = (\beta^\alpha-1)/(\beta-1) =
  \beta^{\alpha-1}+\beta^{\alpha-2}+\cdots+\beta+1\equiv  
  \alpha\pmod 2
\end{equation*}
so that $\alpha=2$. Thus
\begin{equation*}
  \pi^{e} = \beta = 2^{b}-1 = (2^{n}-1)(2^{n}+1), 
\end{equation*}
where both factors of the last term are greater than $1$, as $n > 1$.
If $e = 1$, this  
contradicts the  fact that $\pi$  is a prime.  If $e > 1$,  then $\pi$
divides both $2^{n}-1$ and  $2^{n}+1$, which contradicts the fact that
these two numbers are coprime.



The  other  possibility  is  for  $K$  to  be  the  alternating  group
$\Alt(2^{b})$  of degree  $2^{b}$.   Since the  automorphism group  of
$\Alt(2^{b})$  is $\Sym(2^{b})$,  we  obtain that  $\Gamma$ is  either
$\Alt(2^{b})$ or $\Sym(2^{b})$.  In view of Lemma~\ref{lemma:even},
we have $\Gamma = \Alt(V)$, as claimed.

\newenvironment{acknowledgements}[1][]{\section*{Acknowledgements}#1}%

\begin{acknowledgements}
  The authors are indebted to R{\"u}diger Sparr and Ralph Wernsdorf
  for reading a previous version and suggesting several changes,
  pointing out in particular a serious oversight on our part regarding
  the parity of permutations, and providing a shorter argument for
  Subsection~\ref{subsec:almostsimple}. 
\end{acknowledgements}


\providecommand{\bysame}{\leavevmode\hbox to3em{\hrulefill}\thinspace}
\providecommand{\MR}{\relax\ifhmode\unskip\space\fi MR }
\providecommand{\MRhref}[2]{%
  \href{http://www.ams.org/mathscinet-getitem?mr=#1}{#2}
}
\providecommand{\href}[2]{#2}

\end{document}